\documentclass[11pt]{article}
\usepackage{amssymb}
\usepackage{latexsym,amsmath,amssymb,amsfonts,epsfig,graphicx,cite,psfrag,dsfont,enumitem}
\usepackage{eepic,color,colordvi,amscd}
\usepackage{mathrsfs}
\usepackage{tikz}
\usepackage{subfigure}
\usepackage[english]{babel}
\usepackage[center]{caption2}
\usepackage{amsthm}
\usepackage{multirow}
\usepackage[usenames,dvipsnames]{pstricks}
\usepackage{epsfig}
\usepackage{pst-grad} % For gradients
\usepackage{pst-plot} % For axes
\usepackage[space]{grffile} % For spaces in paths
\usepackage{etoolbox} % For spaces in paths
\usepackage{latexsym,amsmath,amssymb,amsfonts,epsfig,graphicx,cite,psfrag}
\usepackage{eepic,color,colordvi,amscd}
\usepackage{url}

\newtheorem{theo}{Theorem}[section]
\newtheorem{ques}[theo]{Question}
\newtheorem{lem}[theo]{Lemma}

\newtheorem{claim}{Claim}
\newtheorem*{cl}{Claim}

\begin{document}

\title{A strengthening on odd cycles in graphs of given chromatic number}

\author{Jun Gao\thanks{School of Mathematical Sciences, USTC, Hefei, Anhui 230026, China. Email: gj0211@mail.ustc.edu.cn.}
\and
Qingyi Huo\thanks{School of Mathematical Sciences, USTC, Hefei, Anhui 230026, China. Email: qyhuo@mail.ustc.edu.cn.}
\and
Jie Ma\thanks{School of Mathematical Sciences, USTC, Hefei, Anhui 230026, China. Email: jiema@ustc.edu.cn.
Partially supported by National Key Research and Development Project SQ2020YFA070080, NSFC grant 11622110, and Anhui Initiative in Quantum Information Technologies grant AHY150200.}
}

\date{}

\maketitle

\begin{abstract}
Resolving a conjecture of Bollob\'{a}s and Erd\H{o}s, Gy\'{a}rf\'{a}s proved that every graph $G$ of chromatic number $k+1\geq 3$ contains cycles of $\lfloor\frac{k}{2}\rfloor$ distinct odd lengths.
We strengthen this prominent result by showing that such $G$ contains cycles of $\lfloor\frac{k}{2}\rfloor$ consecutive odd lengths.
Along the way, combining extremal and structural tools, we prove a stronger statement that every graph of chromatic number $k+1\geq 7$ contains $k$ cycles of consecutive lengths, except that some block is $K_{k+1}$.
As corollaries, this confirms a conjecture of Verstra\"ete and answers a question of Moore and West.
\end{abstract}

\section{Introduction}

One of the basic results in graph theory says that every graph of chromatic number $k$ contains a cycle of length at least $k$.
This leads to many natural studies on the relation between the chromatic number and the distribution of cycle lengths.
One cornerstone in this direction is the following celebrated theorem, which was conjectured by Bollob\'{a}s and Erd\H{o}s \cite{Erd92} and proved by Gy\'{a}rf\'{a}s \cite{Gyarfas}.

\begin{theo}[Gy\'{a}rf\'{a}s \cite{Gyarfas}]\label{Gyarfas}
Let $k\geq 2$ be an integer.
If $G$ is a graph of chromatic number $k+1$, then $G$ contains cycles of at least $\lfloor\frac{k}{2}\rfloor$ distinct odd lengths.
\end{theo}

This result has inspired extensive research.
Let $k\geq 2$ be an integer and $G$ be a graph of chromatic number $k+1$.
Mihok and Schiermeyer \cite{MS04} obtained an analogue for even cycles that every such $G$ contains cycles of at least $\lfloor\frac{k}{2}\rfloor-1$ distinct even lengths.
Confirming a conjecture of Erd\H{o}s \cite{Erd92}, Kostochka, Sudakov and Verstra\"ete \cite{KSV} showed that
if such $G$ does not contain a triangle, then it contains at least $\Omega(k^2\log k)$ cycles of consecutive lengths.
Recently, the authors and Liu \cite{GHLM} proved a conjecture of Sudakov and Verstra\"ete \cite{SV17}
that such $G$ contains $k-1$ cycles of consecutive lengths.
We remark that Theorem~\ref{Gyarfas} and these results of \cite{MS04,GHLM} are all tight for $G$ being the clique $K_{k+1}$.
For related results, we refer readers to \cite{BV98,HS98,V00,Fan02,SV08,Ma,LM}.

The aim of this note is to provide a common extension of Theorem \ref{Gyarfas} and the aforementioned results of \cite{MS04,GHLM} on cycles of consecutive lengths in graphs of given chromatic number.

\begin{theo}\label{mainresult}
Let $k\geq 6$ be an integer.
If $G$ is a graph of chromatic number $k+1$, then $G$ contains $k$ cycles of consecutive lengths, except that some block of $G$ is $K_{k+1}$.
\end{theo}

Depending on if the graphs contain a triangle or not, we treat the proof of Theorem~\ref{mainresult} differently.
The proof for graphs without a triangle is motivated by \cite{KSV} and utilizes extremal arguments, where we use a new lemma on $A$-$B$ paths (see Lemma \ref{A-B pathmodify}).
On the other hand, the proof for graphs containing a triangle follows the line of \cite{GHLM} and relies on the structural analysis.

As an attempt to have a common generalization of the results of \cite{Gyarfas} and \cite{MS04},
Verstra\"ete conjectured in \cite[Conjecure XVI]{V16} that for any $k\geq 2$, if $G$ is a graph of chromatic number $k+1$,
then $G$ contains $k-1$ cycles of consecutive lengths which start with an odd number.
The case $k=2$ is obvious and the cases $k\geq 6$ follow as a direct corollary of Theorem~\ref{mainresult}.
The following is another result of this paper.

\begin{theo}\label{thmchromatic odd}
Let $k\geq 2$ be an integer.
If $G$ is a graph of chromatic number $k+1$, then there exists some $m$ such that $G$ contains $k-1$ cycles of lengths $2m+1,2m+2,\ldots,2m+k-1$, respectively.
\end{theo}

Here we give a proof for every $k\geq 5$. Unfortunately the proof of Theorem~\ref{thmchromatic odd} for the cases $k=3,4$ requires different techniques and a lengthy argument,
which we present in a separate note and upload as an ancillary file to arXiv.

Using Theorem~\ref{mainresult}, we also can answer a recent question of Moore and West \cite{MW}.
A graph is $k$-{\it critical} if it has chromatic number $k$ but deleting any edge will decrease the chromatic number.
Moore and West \cite[Question 2]{MW} asked whether every $(k+1)$-critical non-complete graph has a cycle of length $2$ modulo $k$.
By Theorem~\ref{mainresult}, we can give an affirmative answer to this question in the following form.

\begin{theo}\label{thmMW}
For $k\geq 6$, every $(k+1)$-critical non-complete graph contains cycles of all lengths modulo $k$.
\end{theo}

We remark that Theorem \ref{thmMW} also holds for $3\leq k\leq 5$.
The case $k=3$ follows by the results in \cite{CS,LY01,S92},
while the complete proof of the cases $k=4,5$ can be found in \cite{H21} which was submitted very recently.
More results related to cycle lengths modulo $k$ can be found in \cite{EH66,Tuza92,CMZ2015,GHLM}.

Returning back to Theorem~\ref{mainresult}, we now give an example to show that
the existence of $k$ cycles of consecutive lengths in graphs of chromatic number $k+1$ which do not contain $K_{k+1}$ is almost tight:
for $k\geq 3$, let $H_k$ be obtained by joining every vertex of the complete graph $K_{k-2}$ to every vertex of the cycle $C_5$.
Then $H_k$ has chromatic number $k+1$ and precisely $k+1$ cycles of consecutive lengths (namely $3, 4, \ldots, k+3$, respectively).
Note that $H_k$ is also $(k+1)$-critical.
It seems plausible that every non-complete $(k+1)$-critical graph contains $k+1$ cycles of consecutive lengths.
Moreover, we wonder if the following much stronger bound on consecutive cycle lengths can hold for $(k+1)$-critical graphs.
%if one can determine the exceptions and carry out more detailed analysis.

\begin{ques}
Let $k\geq 3$ be an integer. Is there a function $f_k(n)$ which goes to infinity as $n$ goes to infinity
such that every $n$-vertex $(k+1)$-critical graph contains $f_k(n)$ cycles of consecutive lengths?
\end{ques}

The rest of the paper is organized as follows.
In Section~\ref{SECpre}, we introduce the notation and some results in \cite{GHLM}.
In Section~\ref{sec:ABpath}, we give a new lemma on $A$-$B$ paths.
In Sections~\ref{SECtriangle} and \ref{SECnotriangle}, we consider graphs containing a triangle and graphs without a triangle of chromatic number at least seven, respectively.
In Section~\ref{SEC678}, we investigate graphs of chromatic number six.
In Section \ref{SECmain}, we complete the proofs of Theorems \ref{mainresult} and \ref{thmchromatic odd}.

\section{Preliminaries}\label{SECpre}

We follow the notation in \cite{LM,GHLM}.
Throughout the paper, we write $[k]$ for the set $\{1,2,...,k\}$ for a positive integer $k$.
Let $G$ be a graph.
For a non-trivial partition $(A,B)$ of $V(G)$,\footnote{A partition $(A,B)$ is {\it non-trivial} if each of $A$ and $B$ is non-empty.}
we say a path $P$ is an {\it $A$-$B$ path} if one end of $P$ is in $A$ and another is in $B$.
Let $P$ be a path in a graph $G$. Let $|P|$ be the number of edges in $P$.
We say $P$ is an {\it even} (respectively, {\it odd}) path if $|P|$ is an even (respectively, odd) number.
Let $C$ be a cycle with vertices $v_0,v_1,\ldots,v_{t-1}$ in cyclic order.
Let $C_{i,j}$ denote the subpath $v_iv_{i+1}\ldots v_j$ of $C$, where the indices are taken under the additive group $\mathbb{Z}_{t}$.

Let $H$ be a subgraph of $G$.
For a vertex $v \in V(G)$, let $N_G(v)$ be the neighborhood of $v$ in $G$.
We say that $H$ and a vertex $v\in V(G)-V(H)$ are {\it adjacent} in $G$ if $v$ is adjacent in $G$ to some vertex in $V(H)$.
We say that $v$ is a neighbor of $H$ if and only if $v$ and $H$ are adjacent.
Let $N_G(H)=\bigcup_{v\in V(H)}N_G(v)-V(H)$ be the neighborhood of $H$ in $G$ and $N_G[H]=N_G(H)\cup V(H)$ be the closed neighborhood of $H$ in $G$.
%We say that $H$ and a vertex $v\in V(G)-V(H)$ are {\it adjacent} in $G$ if $v$ is adjacent in $G$ to some vertex in $V(H)$.
%Let $N_G(H)=\bigcup_{v\in V(H)}N_G(v)-V(H)$ and $N_G[H]=N_G(H)\cup V(H)$.
For $S\subseteq V(G)$, we say that a graph $G'$ {\it is obtained from $G$ by contracting $S$} into a vertex $s$, if $V(G')=(V(G)-S)\cup\{s\}$ and $E(G')=E(G-S)\cup\{vs: v\in V(G)-S$ is adjacent to $S$ in $G\}$.
For two vertex-disjoint subgraphs $H_1$ and $H_2$ of $G$.
Let $N_{H_1}(H_2):=N_G(H_2)\cap V(H_1)$.
A vertex $v$ of a connected graph $G$ is a {\em cut-vertex} of $G$ if $G-v$ contains at least two components.
A {\em block} $B$ in $G$ is a maximal connected subgraph of $G$ such that there is no cut-vertex of $B$.
Note that a block is an isolated vertex, an edge or a $2$-connected graph.
An {\em end-block} in $G$ is a block in $G$ containing at most one cut-vertex of $G$.

We say that $(G,x,y)$ is a {\it rooted graph} if $G$ is a graph and $x,y$ are two distinct vertices of $G$.
The {\it minimum degree} of a rooted graph $(G,x,y)$ is $\min \{d_G(v):v\in V(G)-\{x,y\}\}$.
We also say that a rooted graph $(G,x,y)$ is {\it $2$-connected} if $G+xy$ is $2$-connected.
We say that $k$ paths or $k$ cycles $P_1,P_2,\ldots,P_k$ are {\it admissible} if $|P_1|\geq 2$ and $|P_1|,|P_2|,\ldots,|P_k|$ form an arithmetic progression of length $k$ with common difference one or two.
We need the following results in \cite{GHLM}.

\begin{theo}[\cite{GHLM}, Theorem 3.1]\label{mainthm}
Let $k$ be a positive integer.
If $(G,x,y)$ is a $2$-connected rooted graph of minimum degree at least $k+1$,
then there exist $k$ admissible paths between $x$ and $y$ in $G$.
\end{theo}

Define $K_4^-$ to be the graph obtained from $K_4$ by deleting one edge.

\begin{lem}[\cite{GHLM}, Lemma 5.2]\label{K_3noK_4-}
Let $k \geq 3$ and $G$ be a $3$-connected graph of minimum degree at least $k$.
If $G$ contains a $K_3$ but does not contain a $K_4^-$, then $G$ contains $k$ cycles of consecutive lengths.
\end{lem}

\section{A new lemma on $A$-$B$ paths}\label{sec:ABpath}
The following well-known lemma is due to Bondy and Simonovits \cite{BS} and, independently, Verstra\"ete \cite{V00}.

\begin{lem}[\cite{BS,V00}]\label{A-B path}
Let $G$ be a graph comprising a cycle with a chord and $(A,B)$ be a non-trivial partition of $V(G)$.
Then $G$ contains $A$-$B$ paths of every length less than $|V(G)|$, unless $G$ is bipartite with the bipartition $(A,B)$.
\end{lem}

We prove a modified version of Lemma \ref{A-B path} as follows.

\begin{lem}\label{A-B pathmodify}
Let $G$ be a connected graph of minimum degree at least three and $(A,B)$ be a non-trivial partition of $V(G)$.
For any cycle $C$ in $G$, there exist $A$-$B$ paths of every length less than $|V(C)|$ in $G$, unless $G$ is bipartite with the bipartition $(A,B)$.
\end{lem}

\begin{proof}
We assume that $(A,B)$ is not a bipartition of $G$ (even if $G$ is bipartite). Let $C:=v_0v_1\ldots v_s$.

First consider that $V(C)\subseteq A$.
Since $G$ is connected, there exists a path $Q$ which links some $u\in B$ and a vertex say $v_0$ in $C$ internally disjoint from $V(C)$.
We may assume that $V(Q)-u\subseteq A$ and write $Q=x_0x_1\ldots x_{t}$, where $x_0=v_0$ and $x_t=u$.
It is straightforward to see that $x_ix_{i+1}\ldots x_t$ and $C_{0,j}\cup Q$ for $0\leq i\leq t-1$ and $j\in [s]$ form $A$-$B$ paths of every length less than $|V(C)|+|V(Q)|-1$.

So we may assume that $V(C) \cap B \ne \emptyset$ and similarly, $V(C) \cap A \ne \emptyset$.
Then $(A\cap V(C),B\cap V(C))$ is a non-trivial partition of $V(C)$.
We say that an edge in $G$ is {\it crossing} if one of its endpoints is in $A$ and the other is in $B$, and {\it internal} otherwise.

Suppose that there is an internal edge in $G[V(C)]$.
Then $(A\cap V(C),B\cap V(C))$ is not a bipartition of $G[V(C)]$.
By Lemma~\ref{A-B path}, we may assume that $C$ is an induced cycle in $G$.
So $E(C)$ has an internal edge.
Since $(A\cap V(C),B\cap V(C))$ is non-trivial, there is also a crossing edge in $E(C)$.
Without loss of generality, we can suppose that $v_{s}\in A$ and $v_0,v_1\in B$.
As $\delta(G)\geq 3$ and $C$ is an induced cycle, $v_0$ has a neighbor $x\in V(G)\setminus V(C)$.
If $x\in A$, then one of $C_{j,0}\cup xv_0$ and $C_{j,1}$ is an $A$-$B$ path of length $s-j+2$ for each $2\leq j\leq s$;
otherwise $x\in B$, then one of $xv_0\cup C_{0,j}$ and $C_{s,j}$ is an $A$-$B$ path of length $j+1$ for each $1\leq j\leq s-1$.
Note that a crossing edge is an $A$-$B$ path of length $1$, such an edge exists as $(V(C)\cap A,V(C)\cap B)$ is a non-trivial partition.
Hence in either case, $G$ contains $A$-$B$ paths of every length less than $|V(C)|$.

Therefore, we may assume that every edge in $G[V(C)]$ is crossing.
This shows that $s$ is odd.
Since $(A,B)$ is not a bipartition of $G$, there exists at least one internal edge in $G-E(G[V(C)])$.
By the connectedness of $G$, there exists a path $Q'$ starting with an internal edge $f$ and ending with a vertex say $v_0$ in $C$
such that $Q'$ is internally disjoint from $V(C)$ and all edges in $E(Q')-f$ are crossing.
Then $Q'\cup C_{0,s}$ is a path of length at least $s+1$, whose first edge is internal and all other edges are crossing.
It is not hard to see that $Q'\cup C_{0,s}$ contains subpaths which are $A$-$B$ paths of every length at most $s+1=|V(C)|$.
This finishes the proof of Lemma \ref{A-B pathmodify}.
\end{proof}

\section{Graphs containing a triangle}\label{SECtriangle}

We devote this section to a sharp result on consecutive cycles in graphs containing a triangle.
This improves some results in \cite[Lemmas 5.1 and 5.2 ]{GHLM} to $2$-connected graphs.

\begin{theo}\label{containK3}
	Let $k\geq 2$ be an integer.
	Every $2$-connected graph $G$ of minimum degree at least $k$ containing a triangle $K_3$ contains $k$ cycles of consecutive lengths, except that $G=K_{k+1}$.
\end{theo}

\begin{proof}
	Assume that $G$ is a 2-connected graph with $\delta(G)\geq k$ such that it contains a triangle and $G\neq K_{k+1}$.
	First, suppose that $G$ contains a separating set $S=\{s_1,s_2\}$.
	Let $X$ and $Y$ form a partition of $V(G)-S$, which are separated by $S$ in $G$.
	Let $T_0$ be a $K_3$ in $G$ and denote $V(T_0)$ by $\{u_1,u_2,u_3\}$.
	Since $T_0$ is a clique, we may assume that $T_0$ is contained in $G[X\cup S]$.
	Note that $(G[X \cup S],s_1,s_2)$ is $2$-connected.
	There exist two disjoint paths $L_1,L_2$ from $S$ to $V(T_0)$ internally disjoint from $V(T_0)$ in $G[X\cup S]$.
	Without loss of generality, we may assume that $L_i$ links $s_i$ and $u_i$ for $i=1,2$.
	So $L_1':=L_1\cup u_1u_2\cup L_2,\ L_2':=L_1\cup u_1u_3u_2\cup L_2$ are $2$ paths of consecutive lengths from $s_1$ to $s_2$ in $G[X\cup S]$.
	Also, it is easy to check that $(G[Y\cup S],s_1,s_2)$ is a $2$-connected rooted graph of minimum degree at least $k$.
	Then by Theorem \ref{mainthm}, there exist $k-1$ admissible paths $P_1,P_2,...,P_{k-1}$ from $s_1$ to $s_2$ in $G[Y\cup S]$.
	Concatenating each of these paths with $L_1',L_2'$, we obtain $k$ cycles of consecutive lengths.
	
	Therefore, we may assume that $G$ is $3$-connected.
	So $k\geq 3$.
	If $G$ does not contain a $K_4^-$, then by Lemma \ref{K_3noK_4-}, $G$ contains $k$ cycles of consecutive lengths.
	Therefore, $G$ contains a $K_4^-$.
	Let $T_1$ be a $K_4^-$ in $G$ with $V(T_1)=\{v_1,v_2,v_3,v_4\}$, where $v_1$ has degree two in $T_1$ and $v_2$ is a neighbor of $v_1$.
	Let $K$ be a maximal clique in $G-\{v_1,v_2\}$ containing $v_3,v_4$ and let $t=|V(K)|$.
	We may assume that $t\leq k-1$ (as, otherwise, we can easily find $k$ cycles of lengths $3,4,\ldots,k+2$, respectively).
	
	Suppose that $t=k-1$.
	Since $G$ has minimum degree at least $k$ and $G\neq K_{k+1}$, we may assume that $|V(G)|\geq k+2$.
	Let $H:=G[V(K)\cup\{v_1,v_2\}]$.
	Let $x$ be a vertex in $G-H$.
	Since $G$ is $3$-connected, there exist three internally disjoint paths $M_1,M_2,M_3$ from $x$ to $V(H)$ internally disjoint from $H$.
	Let $y_i$ be the end of $M_i$ in $H$ for $i=1,2,3$.
	It is easy to see that there exists an $i$ and $j$ where $i\ne j$ such that there are $k$ paths $ P_1,\ldots,P_k$ which up to relabelling, having lengths $1,2,\ldots,k$ from $y_i$ to $y_j$ in $H$, respectively.
	Concatenating each of these paths with $M_i \cup M_j$, we obtain $k$ cycles of consecutive lengths.
	
	Therefore, $t\leq k-2$.
	Let $F$ be the component of $G-K$ containing $\{v_1,v_2\}$.
	By the maximality of $K$, every vertex in $G-(K \cup \{v_1,v_2\})$ has at most $t-1$ neighbors in $K$.
	So $\delta((F,v_1,v_2))\geq k-t+1$.
	
	Suppose that $F$ is $2$-connected.
	Then $(F,v_1,v_2)$ is a $2$-connected rooted graph of minimum degree at least $k-t+1$.
	By Theorem \ref{mainthm}, there exist $k-t$ admissible paths $Q_1,Q_2,...,Q_{k-t}$ from $v_1$ to $v_2$ in $F$.
	Note that there exist $t+1$ paths of length $1,2,\ldots,t+1$ from $v_1$ to $v_2$ in $G[K \cup \{v_1,v_2\}]$, respectively.
	Concatenating each of these paths with $Q_i$ for $i\in [k-t]$, we obtain $k$ cycles of consecutive lengths.
	
	Therefore $F$ is not $2$-connected.
	Suppose that there is a leaf $y$ in $F$.
	Then $y$ has at least $k-1$ neighbors in $K$.
	It follows that $|K|\geq k-1$, a contradiction.
	So $|V(F)|\geq3$ and every end-block of $F$ is $2$-connected.
	Let $B$ be an end-block of $F$ with cut-vertex $b$ such that $v_1,v_2\notin V(B)-b$.
	Let $G_1$ be the graph obtained from $G[B\cup (N_G(B-b)\cap K)]$ by contracting $N_G(B-b)\cap K$ into a vertex $w$.
	It is clear that $(G_1,w,b)$ is a $2$-connected rooted graph of minimum degree at least $k-t+2$.
	By Theorem \ref{mainthm}, there exist $k-t+1$ admissible paths from $w$ to $b$ in $G_1$.
	Hence, $G$ contains $k-t+1$ admissible paths $R_i$ from a vertex $p_i\in N_G(B-b)\cap K$ to $b$ for $i\in [k-t+1]$ internally disjoint from $K\cup (F-(B-b))$.
	Let $L$ be a fixed path from $b$ to $\{v_1,v_2\}$ in $F-(B-b)$.
	Without loss of generality, we may assume that $L$ links $b$ and $v_1$.
	Note that there exist $t$ paths from $v_1$ to $p_i$ in $G[K \cup \{v_1,v_2\}]$ with lengths $2,3,\ldots,t+1$, respectively, for each $i\in [k-t+1]$.
	Concatenating these paths with $R_i\cup L$, we obtain $k$ cycles of consecutive lengths.
	This proves Theorem~\ref{containK3}.
\end{proof}

\section{Graphs without a triangle}\label{SECnotriangle}
In this section, we prove the following result on {\it $K_3$-free} graphs.\footnote{A graph is $K_3$-free if it does not contain a triangle as a subgraph.}
Its proof ideas can be traced back to \cite{KSV}.
Our new ingredient is Lemma~\ref{A-B pathmodify}, which assembles the parts of the proof.

\begin{theo}\label{withoutK3}
	Let $k\geq 6$ be an integer.
	If $G$ is a $K_3$-free graph of chromatic number $k+1$, then $G$ contains $k$ cycles of consecutive lengths.
\end{theo}

To facilitate the use of Lemma~\ref{A-B pathmodify}, we need the following lemma.

\begin{lem}\label{longcycle}
	Let $k\geq 3$ be an integer. Let $G$ be a $2$-connected graph of minimum degree at least $k$.
	If $G$ is $K_3$-free, then $G$ contains a cycle of length at least $2k+2$, except that $G=K_{k,n}$ for some $n\geq k$.
\end{lem}

\begin{proof}
	Suppose to the contrary that there exists a 2-connected graph $G$ with $\delta(G)\geq k$,
	which is not $K_{k,n}$ for any $n\geq k$ and does not contain a cycle of length at least $2k+2$.
	We first prove some properties about the longest paths in $G$.
	Let $P=v_0v_1\ldots v_{\ell}$ be any longest path in $G$ which has length $\ell$.
	
	\begin{claim}\label{neighbors}
		The neighborhood of $v_0$ is contained in $V(P)$. Further, $N_G(v_0)$ consists of $k$ vertices such that either (a) $\ell\geq 2k-1$ and $N_G(v_0)=\{v_{2i+1}: 0\leq i\leq k-1\}$,
		or (b) $\ell\geq 2k$ and $N_G(v_0)=\{v_1,v_3,\ldots, v_{s-3},v_s,v_{s+2},\ldots, v_{2k}\}$ for some even integer $4\leq s\leq 2k$.
	\end{claim}
	
	\begin{proof}[Proof of Claim \ref{neighbors}]
		Since $P$ is a longest path in $G$, it is clear that $N_G(v_0)\subseteq V(P)$.
		Let $v_{p_1},v_{p_2},\ldots, v_{p_t}$ be all neighbors of $v_0$ in $P$,
		where $1=p_1<p_2<\cdots< p_t$ and $t\geq k$.
		Since $G$ is $K_3$-free, $p_{i+1}-p_{i}\geq 2$ for each $i$.
		Also by the assumption we see $p_t\leq 2k$.
		The result now follows by an easy analysis. \end{proof}
	
	\begin{claim}\label{endsarenotadjacent}
		There is no cycle of length $\ell+1$ in $G$. Hence, $v_0v_\ell\notin E(G)$ and $\ell\geq 2k$.
	\end{claim}
	
	\begin{proof}[Proof of Claim \ref{endsarenotadjacent}]
		Suppose to the contrary that there is a cycle $C:=x_0x_1\ldots x_{\ell}x_0$,
		where the indices are taken under the additive group $\mathbb{Z}_{\ell+1}$.
		Since $C_{0,\ell}$ is a longest path in $G$ and $G$ is connected,
		we have $V(G)=V(C)$ (otherwise one can find a longer path, a contradiction).
		As $v_0v_\ell\in E(G)$, by Claim \ref{neighbors}, we see $2k-1\leq\ell\leq 2k$.
		Suppose that $\ell=2k-1$.
		For each $j\in\mathbb{Z}_{\ell+1}$,
		the longest path  $C_{j,j-1}$ always has Claim \ref{neighbors}(a) occur.
		So $N_G(x_j)=\{x_{j+2i-1}: i\in [k]\}$.
		This shows that $G$ is a complete bipartite graph $K_{k,k}$, a contradiction.
		
		Therefore $\ell=2k$. Consider the longest path $C_{j,j-1}$ for each $j$.
		Since $x_jx_{j-1}\in E(G)$, Claim \ref{neighbors}(b) must occur for the endpoint $x_j$ of $C_{j,j-1}$.
		So there exists some even integer $s_j\in \{4,5,\ldots,2k\}$
		such that $N_G(x_j)=\{x_{j+1},x_{j+3},\ldots, x_{j+s_j-3},x_{j+s_j},x_{j+s_j+2},\ldots, x_{j+2k}\}$
		where additions are taken under $\mathbb{Z}_{\ell+1}$.
		In particular, there exist some $x_t, x_{t+3}$ which have a common neighbor say $x_m$.
		Without loss of generality, we may assume $t=0$.
		Since $G$ is $K_3$-free, $x_0$ cannot be adjacent to $x_3$.
		This implies $s_0=4$ and thus $x_0$ is adjacent to both of $x_1, x_4$.
		By the same argument, we can further derive that $s_1=4$ and $x_1$ is adjacent to both of $x_2, x_5$.
		Continuing this, we conclude that $s_j=4$ for each $0\leq j\leq 2k$.
		This shows $x_0x_{\ell-3}\in E(G)$.
		However, as $s_0=4$, by Claim \ref{neighbors} we also see $x_0x_{\ell-2}\in E(G)$.
		Then $x_0x_{\ell-2}x_{\ell-3}$ forms a $K_3$ in $G$, a contradiction.
		This proves the claim.
	\end{proof}
	
	For a longest path $P=v_0v_1\ldots v_{\ell}$,
	we call $(v_{\alpha},v_{\beta})_P$ a {\it crossing pair},
	if $v_{\alpha}\in N_G(v_{\ell}), v_{\beta}\in N_G(v_0)$ and $v_{\gamma}\notin N_G(v_0)\cup N_G(v_{\ell})$ for each integer $\gamma\in (\alpha,\beta)$.
	Let $\beta-\alpha$ denote the {\it gap} of a crossing pair $(v_{\alpha},v_{\beta})_P$.
	
	\begin{claim}\label{crossing-pair}
		Each longest path $P$ has a crossing pair, and the gap of each crossing pair of $P$ is at least 2.
	\end{claim}
	\begin{proof}[Proof of Claim \ref{crossing-pair}]
		First, we show the second assertion.
		Let $(v_{\alpha},v_{\beta})_P$ be any crossing pair of $P$.
		Clearly we have $\beta-\alpha\geq 1$.
		If $\beta-\alpha=1$, then one can easily find a cycle of length $\ell+1$, a contradiction to Claim \ref{endsarenotadjacent}.
		Hence, the gap of each crossing pair of $P$ is at least two.
		
		Let $p$ be the maximum integer with $v_pv_0\in E(G)$ and $q$ be the minimum integer with $v_qv_\ell\in E(G)$.
		It will suffice to show that $p>q$.
		First suppose that $p<q$.
		By Claim~\ref{neighbors}, there are two vertex disjoint cycles $C_1,C_2$ of length at least $2k$ in $G$.
		Since $G$ is $2$-connected,
		there are two disjoint paths $L_1, L_2$ from $V(C_1)$ to $V(C_2)$.
		Then it is easy to find two cycles $D_1, D_2$ with $|D_1|+|D_2|=|C_1|+|C_2|+2(|L_1|+|L_2|)\geq 4k+4$,
		which gives a cycle of length at least $2k+2$, a contradiction.
		Now suppose $p=q$.
		Since $G$ is $2$-connected, there is a path $L$ from $v_s\in\{v_0,v_1,\ldots,v_{p-1}\}$ to $v_{t}\in\{v_{q+1},v_{q+2},\ldots,v_{\ell}\}$ in $G$ internally disjoint from $P-v_p$.
		Let $r\in (s,p]$ be the minimum integer with $v_0v_r\in E(G)$.
		Let $C_1'=v_0v_1\ldots v_sLv_tv_{t+1}\ldots v_{\ell}v_pv_{p-1}\ldots v_rv_0$ and $C_2'=v_0v_1\ldots v_sLv_tv_{t-1}\ldots v_rv_0$ be two cycles.
		By Claim~\ref{neighbors}, each of $C_1'$ and $C_2'$ contains at least $2k-1$ vertices in $\{v_0,\ldots, v_p\}$.
		Thus, one of $C_1'$ and $C_2'$ contains at least $(2k-1)+k\geq 2k+2$ vertices, where $k\geq 3$.
		This contradiction finishes the proof.
	\end{proof}

	\begin{claim}\label{nottwo}
		Let $(v_{\alpha},v_{\beta})_P$ be an arbitrary crossing pair of $P$. Then the gap of $(v_{\alpha},v_{\beta})_P$ equals two.
	\end{claim}
	
	\begin{proof}[Proof of Claim \ref{nottwo}]
		We have $\beta-\alpha\geq 2$ from the previous claim and $\beta-\alpha\leq 3$ from the definition of crossing pairs and Claim \ref{neighbors}.
		Let us suppose for a contradiction that $\beta-\alpha=3$.
		Note that this forces that both $v_{0}$ and $v_{\ell}$ satisfy Claim \ref{neighbors}(b).
		So $\alpha$ is odd, $\beta$ is even and $\ell-\beta$ is odd, implying that $\ell$ is odd.
		The cycle $v_0v_1\ldots v_{\alpha}v_{\ell}v_{\ell-1}\ldots v_{\beta}v_0$ has length $\ell-1$,
		so $\ell-1\leq 2k+1$, i.e., $\ell\leq 2k+2$.
		By Claim \ref{endsarenotadjacent}, we also see $\ell\geq 2k$.
		Combining with the above, we can conclude that $\ell=2k+1$.
		Therefore, $N_G(v_0)=N_G(v_\ell)=\{v_1,v_3,\ldots,v_{\alpha},v_{\beta},v_{\beta+2},\ldots, v_{2k}\}:=N$, where $\beta=\alpha+3$ and $|N|=k$.
		
		Consider the longest path $v_{\alpha+2}v_{\alpha+1}\ldots v_0v_{\beta}v_{\beta+1}\ldots v_{\ell}$.
		If $v_{\alpha+2}$ is adjacent to $v_{\alpha-1}$, then using $v_\ell v_\alpha\in E(G)$, we can find a cycle of length $\ell+1$, a contradiction.
		Hence by Claim~\ref{neighbors}, $N_G(v_{\alpha+2})=\{v_{\alpha+1}, v_{\alpha-2}, \ldots, v_1, v_\beta, v_{\beta+2}, \ldots, v_{2k}\}=(N\backslash \{v_\alpha, v_\beta\})\cup \{v_{\alpha+1},v_\beta\}$.
		Similarly, by considering the path $v_{\alpha+1}v_{\alpha+2}\ldots v_\ell v_{\alpha}v_{\alpha-1}\ldots v_{0}$,
		we can derive $N_G(v_{\alpha+1})=(N\backslash \{v_\alpha, v_\beta\})\cup \{v_{\alpha},v_{\alpha+2}\}$.
		Note that $|N|=k\geq 3$. So $v_{\alpha+1}$ and $v_{\alpha+2}$ have a common neighbor, which forces a triangle in $G$.
		This proves Claim~\ref{nottwo}.
	\end{proof}
	
	Therefore $\beta-\alpha=2$.
	Note that $v_0v_1\ldots v_{\alpha}v_{\ell}v_{\ell-1}\ldots v_{\beta}v_0$ is a cycle of length $\ell$.
	Hence, we have $2k\leq \ell\leq 2k+1$.
	First suppose that $\ell=2k+1$.
	If $v_0$ satisfies Claim \ref{neighbors}(b), then there exists $s$ such that $v_{s-3},v_s$ are neighbors of $v_0$ for some even number $4\leq s\leq 2k$.
	By Claim \ref{neighbors}, since $s\geq 4$, $v_{\ell}$ must have a neighbor in $\{v_{s-3},v_{s-2},v_{s-1}\}$.
	By Claim \ref{endsarenotadjacent}, $v_{\ell}$ is not adjacent to $v_{s-1}$ and $v_{s-4}$.
	Since $G$ is $K_3$-free, exactly one of $v_{s-3}$ and $v_{s-2}$ can be adjacent to $v_{\ell}$.
	By Claim \ref{nottwo}, $v_{\ell}$ is not adjacent to $v_{s-3}$.
	Therefore, $v_{\ell}v_{s-2}\in E(G)$.
	Since $G$ is $K_3$-free, $v_{s-1}$ is not adjacent $v_{s-3}$.
	Suppose $v_{s-1}$ is adjacent to $v_{s-4}$, then $v_0v_1\ldots v_{s-4}v_{s-1}\ldots v_{\ell}v_{s-2}v_{s-3}v_{0}$ is a cycle of length $\ell+1$, contradicting Claim \ref{endsarenotadjacent}.
	So, both $v_{s-1}$ and $v_{\ell}$ are not adjacent to $v_{s-4}$ and $v_{s-3}$.
	Then the longest path $v_{s-1}v_{s-2}\ldots v_0v_{s}v_{s+1}\ldots v_{\ell}$ would contain a crossing pair of gap at least $3$, contradicting Claim \ref{nottwo}.
	Therefore, by symmetry, we may assume that both $v_0$ and $v_{\ell}$ satisfy Claim \ref{neighbors}(a).
	So, we have $N_G(v_0)=\{v_1,v_3,\ldots, v_{2k-1}\}$ and $N_G(v_{\ell})=\{v_{\ell-1},v_{\ell-3},\ldots, v_{\ell-2k+1}\}$.
	Then $v_0v_1\ldots v_{2k-2}v_{2k+1}v_{2k}v_{2k-1}v_0$ is a cycle of length $\ell+1$, contradicting Claim \ref{endsarenotadjacent}.
	
	Hence we have $\ell=2k$.
	In this case, $N_G(v_0)=N_G(v_{2k})=\{v_1, v_3, \ldots, v_{2k-1}\}:=Y$. Let $X=V(P)\backslash Y$.
	For each even integer $2\leq j\leq 2k-2$, $R_j:=v_jv_{j-1}\ldots v_0v_{j+1}v_{j+2}\ldots v_{2k}$ is a longest path in $G$.
	By Claim \ref{neighbors}(a), we have $N_G(v_j)=Y$ for each even $j$.
	As $G$ is $K_3$-free, $G[V(P)]$ consists of a copy $K_{k,k+1}$ with two parts $X$ and $Y$.
	For every vertex $w$ of $G-V(P)$, if $w$ has a neighbor in $X$, then there is a path of length $2k+1$ in $G$, a contradiction. Hence,  $N_{V(P)}(G-V(P))\subseteq Y$. Suppose that $G-V(P)$ contains an edge $e$. Since $G$ is $2$-connected, there exist two vertex disjoint paths $T_1$ and $T_2$ from $V(P)$ to $V(e)$ internally disjoint from $V(P)$. Since $N_{V(P)}(G-V(P))\subseteq Y$ and G[V(P)] is a complete bipartite graph. Without loss of generality, we may assume that $V(T_1)\cap V(P)=\{v_1\}$ and $V(T_2)\cap V(P)=\{v_3\}$. Then $v_0v_1\cup T_1\cup e \cup T_2 \cup v_3v_4\ldots v_{2k}$ is a path of length at least $2k+1$, a contradiction. Therefore, $G-V(P)$ forms an independent set and every vertex of $G-V(P)$ can only be adjacent to vertices in $Y$.
	Since $\delta(G)\geq k$ and $|Y|=k$, we see that $N_G(v)=Y$ for each $v\in V(G)-Y$.
	Hence $G$ is a complete bipartite graph $K_{k,m}$ for some $m$, a contradiction.
	This completes the proof of Lemma \ref{longcycle}.
\end{proof}

We remark that Lemma \ref{longcycle} is best possible by the following examples.
For any integers $k\geq 3$ and $m\geq 2k$,
let $K$ denote a complete bipartite graph $K_{k-1,m}$ with two parts $X$ and $Y$, where $|X|=k-1$ and $Y=Y_1\cup Y_2$ has size $m$ with $|Y_i|\geq k$ for $i\in \{1,2\}$.
Let $G_{k,m}$ be the graph obtained from $K$ by adding two new vertices $x_1, x_2$ and edges in $\{x_1x_2, x_1u, x_2v: \forall u\in Y_1, \forall v\in Y_2\}$.
We see that $G_{k,m}$ is a $2$-connected $K_3$-free graph of minimum degree at least $k$, whose longest cycles have length $2k+2$.

\medskip

Now we are ready to prove Theorem~\ref{withoutK3}.
Let $T$ be a tree with root $r$.
For $a,b\in V(T)$, let $T_{a,b}$ be the unique path between $a$ and $b$ in $T$.

\begin{proof}[\bf Proof of Theorem~\ref{withoutK3}]
	Let $k\geq 6$ and $G$ be a $K_3$-free graph of chromatic number $k+1$.
	Fix a vertex $r$ and let $T$ be the breadth first search tree in $G$ with root $r$.
	Let $L_0=\{r\}$ and $L_i$ be the set of vertices of $T$ at distance $i$ from its root $r$ for $i\geq 1$.
	There exists some $t\geq 1$ such that $G[L_t]$ has chromatic number at least $\ell:=\lceil (k+1)/2\rceil$ where $\ell\geq 4$.
	Let $H$ be a $\ell$-critical subgraph of $G[L_t]$.
	So $H$ is a $2$-connected non-bipartite $K_3$-free graph of minimum degree at least $\ell-1\geq 3$.
	By Lemma \ref{longcycle}, $H$ contains a cycle of length at least $2\ell$.
	Let $T'$ be the minimal subtree of $T$ whose set of leaves is precisely $V(H)$, and let $r'$ be the root of $T'$.
	Let $h$ denote the distance between $r'$ and vertices in $H$ in $T'$.
	Since $G$ is $K_3$-free, $h\geq 2$.
	By the minimality of $T'$, $r'$ has at least two children in $T'$.
	Let $x$ be one of its children.
	Let $A$ be the set of vertices in $H$ which are the descendants of $x$ in $T'$ and let $B=V(H)-A$.
	Then both $A,B$ are nonempty and for any $a\in A$ and $b\in B$, $T_{a,b}$ has the same length $2h$.
	By Lemma \ref{A-B pathmodify}, there are $2\ell-1$ paths of $H$ from a vertex of $A$ to a vertex of $B$ of length $1,2,\ldots, 2\ell-1$, respectively.
	Putting all together, we see that $G$ contains $2\ell-1=2\lceil (k+1)/2\rceil-1\geq k$ cycles of consecutive lengths.
\end{proof}

\section{Graphs of chromatic number $6$}\label{SEC678}
As a further exploration of the proof of Theorem~\ref{withoutK3}, we now consider consecutive cycles in graphs of chromatic number $6$.
The following is the main result of this section.

\begin{theo}\label{678}
Every graph of chromatic number six contains four cycles of consecutive lengths which start with an odd number.
\end{theo}

\begin{proof}
It suffices to consider $6$-critical graphs $G$.
Suppose that $G$ does not contain four cycles of lengths $2m+1, 2m+2, 2m+3$, and $2m+4$ for any integer $m$.
Since $G$ is 2-connected with $\delta(G)\geq 5$, by Theorem \ref{containK3}, we may assume that $G$ is $K_3$-free.
Fix a vertex $r$ and let $T$ be the breadth first search tree in $G$ with root $r$.
Let $L_0=\{r\}$ and $L_i$ be the set of vertices of $T$ at distance $i$ from its root $r$.

We first show that every component of $G[L_i]$ for $i\geq0$ has chromatic number at most 3.
Suppose to the contrary that there exists a component $D$ of $G[L_t]$ which has chromatic number at least $4$ for some $t$.
Then using the exactly same arguments as in the proof of Theorem~\ref{withoutK3} (taking $\ell=4$ therein).\footnote{ We remark that this method applies only to $4$-critical graphs, as Lemma~\ref{A-B pathmodify} does not hold for $3$-critical graphs.}
one can derive that $G$ contains $2\ell-1=7$ cycles of consecutive lengths, a contradiction to our assumption.

We now prove a claim which is key for this proof.
For a connected graph $D$, a vertex $x\in V(D)$ is called {\it good} if it is not contained in the minimal connected subgraph of $D$ which contains all $2$-connected blocks of $D$, and {\it bad} otherwise.

\begin{cl}
Let $H_1$ be a non-bipartite component of $G[L_i]$ and $H_2$ be a non-bipartite component of $G[L_{i+1}]$ for some $i\geq 1$. If $N_{H_1}(H_2)\neq\emptyset$,
then every vertex in $N_{H_1}(H_2)$ is a good vertex of $H_1$.
\end{cl}

\begin{proof}[Proof of Claim]
Suppose that there exists a bad vertex $v$ of $H_1$ which has a neighbor $u$ in $H_2$.
Let $T'$ be the minimal subtree of $T$ whose set of leaves is precisely $V(H_1)$,
and let $r'$ be the root of $T'$.
Let $h$ denote the distance between $r'$ and vertices in $H_1$ in $T'$.
Since $G$ is $K_3$-free, $h \geq 2$.
By the minimality of $T'$, $r'$ has at least two children in $T'$.
Fix a child $x$ of $r'$ in $T'$ and let $Y$ be the set of the children of $r'$ in $T'$ other than $x$.
Let $A$ be the set of vertices in $H_1$ which are the descendants of $x$ in $T'$ and let $B=V(H_1)-A$. Note that every vertex in $B$ is a descendant of a vertex in $Y$ in $T'$.
Let $A'$ be the set of vertices in $L_i-A$ which are the descendants of $x$ in $T$.
Let $B'$ be the set of vertices in $L_i-B$ which are the descendants of $Y$ in $T$.
Let $M:=L_i-(A\cup A'\cup B\cup B')$.
Note that $A,A',B,B'$ and $M$ form a partition of $L_i$.

Let $C=v_0v_1\ldots v_n$ be an odd cycle of $H_1$, where $n\geq 4$.
Suppose that $V(C)\subseteq A$.
Let $b$ be a vertex in $B$.
Since $H_1$ is connected, there exists a path $P$ from $b$ to $V(C)$ internally disjoint from $V(C)$.
Without loss of generality, we assume that $V(P)\cap V(C)=\{v_0\}$.
Then $P\cup C_{0,i}\cup T_{b,v_i}$ for $i=0,1,...,4$ gives $5$ cycles of consecutive lengths, a contradiction.
Therefore, $B\cap V(C)\neq\emptyset$, and similarly, $A\cap V(C)\neq\emptyset$.
Then there must be an $A$-$B$ path of length $4$ in $C$
(otherwise, since $4$ and $|C|$ is co-prime and $|C|\geq 5$, one can deduce that all vertices of $C$ are contained in one of the two parts $A$ and $B$, a contradiction).

Since $C$ is an odd cycle, we may assume that $v_0,v_1\in A$ and $v_2\in B$.
Then $T_{v_1,v_2}\cup v_2v_1,\ T_{v_0,v_2}\cup v_2v_1v_0$ are two cycles of lengths $2h+1,2h+2$, respectively.
We have showed that there exists some $A$-$B$ path of length $4$ in $C$ which gives a cycle of length $2h+4$,
so we may assume that there is no $A$-$B$ path of length $3$ in $C$.
This would force that $v_{3i}, v_{3i+1}\in A$ and $v_{3i+2}\in B$ for each possible $i\geq 0$.
So $|C|\geq 9$ and $G$ contains a cycle of length $\ell\in\{2h+1,2h+2,2h+4,2h+5,2h+7,2h+8\}$.
Moreover, for any path $P'=u_0u_1\ldots u_m$ in $H_1$ with $u_0\in V(C)$ and $(V(P)-u_0)\cap V(C)=\emptyset$,
we can derive that $u_j\in B$ if $j\equiv 0$ modulo $3$ and $u_j\in A$ if $j\equiv 1$ or $2$ modulo $3$; call this property $(\star)$.
In particular, since $H_1$ is connected, for any vertex $b\in B$, there exists a path of length 2 in $H_1$ from $b$ to some vertex in $A$.

\medskip
\noindent\textbf{Case 1.} The component $H_2$ has a neighbor in $M$.

Note that every vertex of $H_2$ has a neighbor in $L_i$.
Suppose that there exists a vertex $c\in V(H_2)$ which has a neighbor $c'$ in $M$.
Recall that $v$ is a bad vertex in $H_1$ and let $u\in N_{H_2}(v)$.
Clearly there exists a path $z_1z_2z_3z_4z_5$ of length $4$ in $H_1$ with $z_1=v$.
It is easy to see that $T_{z_i,c}$ contains $r'$ for $i\in[5]$, so they have the same length.
Let $P''$ be a fixed path from $u$ to $c$ in $H_2$.
Then $P''\cup uz_1z_2\ldots z_i \cup T_{z_i,c'}\cup cc'$, for $i\in[5]$ are $5$ cycles of consecutive lengths in $G$, a contradiction.
%Therefore $N_M(H_2)=\emptyset$.

\medskip
\noindent\textbf{Case 2.} The component $H_2$ has a neighbor in $B\cup B'$ where $N_M(H_2)=\emptyset$.

Suppose that $N_{B\cup B'}(H_2)\neq\emptyset$.
If $N_{A\cup A'}(H_2)\neq\emptyset$, then since $H_2$ is connected and every vertex of $H_2$ has a neighbor in $A\cup A'\cup B\cup B'$,
there exist two adjacent vertices $p,q$ of $H_2$ such that $p$ has a neighbor $p'$ in $A\cup A'$ and $q$ has a neighbor $q'$ in $B\cup B'$.
Then $p'pqq'\cup T_{p',q'}$ is a cycle of length $2h+3$.
It follows that $G$ contains $4$ cycles of lengths $2h+1,2h+2,2h+3,2h+4$, respectively.
Therefore $N_{A\cup A'}(H_2)=\emptyset$.
Since $N_{A\cup B}(H_2)\neq\emptyset$, we have that $v\in B$.
Let $u$ be any vertex in $N_{H_2}(v)$.
Choose $w_1\in V(H_2)$ such that there exists a path $Q$ of length $2$ from $u$ to $w_1$ in $H_2$.
Let $w_2$ be a neighbor of $w_1$ in $B\cup B'$.
Suppose that $w_2\neq v$.
We have showed that there exists a path $R$ of length $2$ in $H_1$ from $v$ to some vertex say $v'$ in $A$.
Then $R \cup vu \cup Q \cup w_1w_2\cup T_{w_2,v'}$ is a cycle of length $2h+6$.
So $G$ contains cycles of lengths $2h+5,2h+6,2h+7,2h+8$, a contradiction.
Therefore $w_2=v$ and $w_1\in N_{H_2}(v)$. That says, every vertex in $H_2$ of distance $2$ from a neighbor of $v$ is a neighbor of $v$.
Continuing to apply this along with a path from $u$ to an odd cycle $C_0$ in $H_2$,
we could obtain that $v$ is adjacent to all vertices of $C_0$, which contradicts that $G$ is $K_3$-free.
%Therefore, $N_{B\cup B'}(H_2)=\emptyset$.

\medskip
\noindent\textbf{Case 3.} Neighbors of $H_2$ which belong to $L_i$ belong to $A\cup A'$.

Now we see that $N_{L_i}(H_2)\subseteq A\cup A'$.
This forces that $v\in A$.
For any neighbor $u$ of $v$ in $H_2$,
let $w_3\in V(H_2)$ satisfies that there exists a path $Q'$ of length $2$ from $u$ to $w_3$ in $H_2$.
By property $(\star)$ and the fact that $v\in A$ is bad in $H_1$, we can infer that there exists a path $t_3vt_1t_2$ in $H_1$ such that $t_1\in A$ and $t_2,t_3\in B$.
Note that $v$ and $t_1$ are symmetric.
Let $w_4$ be a neighbor of $w_3$ in $A\cup A'$.
Suppose that $w_4\notin\{v,t_1\}$.
Then $vu\cup Q'\cup w_3w_4\cup T_{w_4,t_2}\cup t_2t_1v$ is a cycle of length $2h+6$.
So again, $G$ contains cycles of lengths $2h+5,2h+6,2h+7,2h+8$, a contradiction.
Therefore, $w_4\in\{v,t_1\}$. That is, every vertex in $H_2$ of distance $2$ from a neighbor of $v$ or $t_1$ is adjacent to one of $v,t_1$.
Continuing to apply this along with a path from $u$ to an odd cycle $C_1$ in $H_2$,
we could obtain that every vertex of $C_1$ is adjacent to one of $v,t_1$.
But this would force a $K_3$ in $G$.
This final contradiction completes the proof of this claim.
\end{proof}

Now, we define a coloring $c:V(G)\rightarrow \{1,2,3,4,5\}$ as following.
Let $D$ be any bipartite component of $G[L_i]$ for some $i$.
If $i$ is even, we color one part of $D$ with color $1$ and the other part with color $2$,
and if $i$ is odd, we color one part of $D$ with color $4$ and the other part with color $5$.
Let $F$ be any non-bipartite component of $G[L_j]$ for some $j$.
If $j$ is even, by using the block structure of $F$, we can properly color $V(F)$ with colors $1,2$ and $3$ by coloring bad vertices with colors $1,2$ and $3$ and coloring good vertices with colors $1$ and $2$.
If $j$ is odd, then we also can properly color $V(F)$ with colors $3, 4$ and $5$ by coloring bad vertices with colors $3,4$ and $5$ and coloring good vertices with colors $4$ and $5$.

Next, we argue that $c$ is a proper coloring on $G$.
Let $H_1$ be a component of $G[L_i]$ and $H_2$ be a component of $G[L_{i+1}]$ for $i\geq0$ such that there exists an edge between $H_1$ and $H_2$.
If one of them is bipartite, then $c$ is proper on $V(H_1)\cup V(H_2)$ .
Therefore, both $H_1$ and $H_2$ are non-bipartite.
By the above claim, all vertices of $H_2$ are not adjacent to vertices of color $3$ in $H_1$.
It follows that $c$ is proper on $V(H_1)\cup V(H_2)$.
Therefore, $c$ is a proper $5$-coloring of $G$, which contradicts that $G$ is $6$-critical.
This completes the proof of Theorem \ref{678}.
\end{proof}

\section{Proofs of Theorems \ref{mainresult} and \ref{thmchromatic odd}}\label{SECmain}

We conclude with the proofs of Theorems \ref{mainresult} and \ref{thmchromatic odd}.

\begin{proof}[\bf Proof of Theorem \ref{mainresult}]
Let $G'$ be a $(k+1)$-critical subgraph of $G$.
Then $G'$ is a $2$-connected graph of minimum degree at least $k$.
If $G'$ is $K_3$-free, then by Theorem \ref{withoutK3}, $G'$ contains $k$ cycles of consecutive lengths.
If $G'$ contains a $K_3$, then by Theorem \ref{containK3}, either $G'$ contains $k$ cycles of consecutive lengths or $G'$ is $K_{k+1}$.
In the latter case, let $B$ be the block of $G$ containing $G'$.
If there exists $x\in V(B)\backslash V(G')$,
then there are two internally disjoint paths from $x$ to two vertices in $G'$ and then we can easily find $k$ cycles of consecutive lengths.
Therefore, the block $B$ of $G$ is a copy of $K_{k+1}$.
\end{proof}

\begin{proof}[\bf Proof of Theorem \ref{thmchromatic odd}.]
This follows from Theorems \ref{mainresult} and \ref{678}.
\end{proof}

\end{document}